\documentclass[12pt]{amsart}

\usepackage{amsmath}
\usepackage{amssymb}
\usepackage{amsfonts}
\usepackage{amsthm}
\usepackage{enumerate}
\usepackage{hyperref}
\usepackage{color}

\theoremstyle{plain}
\newtheorem{thm}{Theorem}[section]

\newtheorem{prop}[thm]{Proposition}
\newtheorem{lem}[thm]{Lemma}
\newtheorem{cor}[thm]{Corollary}

\theoremstyle{definition}
\newtheorem{dfn}[thm]{Definition}

\newtheorem{dfns-rems}[thm]{Definitions and Remarks}
\newtheorem{notas-rems}[thm]{Notations and Remarks}
\newtheorem{exmps-rems}[thm]{Examples and Remarks}

\begin{document}


\title[Regularity of Powers of edge ideal of very well-covered graphs]{Regularity of Powers of edge ideal of very well-covered graphs}


\author[P. Norouzi]{P. Norouzi}

\address{P. Norouzi, Science and Research Branch Islamic Azad University (IAU),
Tehran, Iran.}

\email{norouzi.pooran@yahoo.com}

\author[S. A. Seyed Fakhari]{S. A. Seyed Fakhari}

\address{S. A. Seyed Fakhari, School of Mathematics, Statistics and Computer Science,
College of Science, University of Tehran, Tehran, Iran.}

\email{aminfakhari@ut.ac.ir}

\urladdr{http://math.ipm.ac.ir/$\sim$fakhari/}

\author[S. Yassemi]{S. Yassemi}

\address{S. Yassemi, School of Mathematics, Statistics and Computer Science,
College of Science, University of Tehran, Tehran, Iran.}

\email{yassemi@ut.ac.ir}

\urladdr{http://math.ipm.ac.ir/$\sim$yassemi/}


\begin{abstract}
Let $k\geq 3$ be an integer and $G$ be a very well-covered graph with ${\rm odd-girth}(G)\geq 2k+1$. Assume that $I(G)$ is the edge ideal of $G$. We show that for every integer $s$ with $1\leq s\leq k-2$, we have ${\rm reg}(I(G)^s)=2s+\nu(G)-1$, where $\nu (G)$ is the induced matching number of $G$.
\end{abstract}






\maketitle


\section{Introduction} \label{sec1}
Let $I$ be a homogeneous ideal in the polynomial ring $R = \mathbb{K}[x_1, \ldots, x_n]$. Suppose that the minimal free resolution of $I$ is given by
$$0\rightarrow \cdots \rightarrow \oplus_j R(-j)^{\beta _{1,j}(I)}\rightarrow \oplus_j R(-j)^{\beta _{0,j}(I)}\rightarrow I\rightarrow 0.$$
The Castelnuovo-Mumford regularity (or simply, regularity) of $I$, denoted by ${\rm reg}(I)$, is defined as
$${\rm reg}(I)={\rm max}\{j-i\mid \beta _{i,j}(I)\neq 0\},$$
and is an important invariant in commutative algebra and algebraic geometry.

There is a natural correspondence between quadratic squarefree monomial ideals of $R$ and finite simple graphs with $n$ vertices. To every simple graph $G$ with vertex set $V(G)=\{x_1, \ldots, x_n\}$ and edge set $E(G)$, we associate an ideal $I=I(G)$ defined by$$I(G)=(x_ix_j: \{x_i, x_j\}\in E(G))\subseteq R.$$

Computing and finding bounds for the regularity of powers of edge ideals have been investigated by a number of researchers (see for example \cite{aa}, \cite{abs}, \cite{ab}, \cite{sb}, \cite{ac}, \cite{cf}, \cite{msy}).
Cutkosky, Herzog, Trung, \cite{sc}, and independently  Kodiyalam \cite{vk}, prove that for a homogenous ideal $I$ in a polynomial ring, reg$(I^s)$ is a linear function for $s\gg0$, i.e, there exist integers $a$, $b$, $s_0$ such that $${\rm reg} (I)^s=as+b \ \ \ \ {\rm for \ all} \ s\geq s_0.$$

It is known that $a$ is bounded above by the maximum degree of elements in a minimal generating set of $I$. But a general bound for $b$ as well as $s_0$ is unknown.

Katzman \cite{mk}, proves that for any graph $G$,
\[
\begin{array}{rl}
{\rm reg}(I(G))\geq \nu(G)+1,
\end{array} \tag{$\ast$} \label{ast}
\]
where $\nu(G)$ denotes the induced matching number of $G$. Beyarslan, H${\rm \grave{a}}$ and Trung \cite{sb}, generalize Katzman's inequality by showing that$${\rm reg}(I(G)^s)\geq 2s+\nu(G)-1,$$for every integer $s\geq 1$. In the same paper, the authors prove the equality for every $s\geq 1$, if $G$
is a forest and for every $s\geq 2$, if $G$ is a cycle (see \cite[Theorems 4.7 and 5.2]{sb}). In \cite{msy}, the authors show that the equality $${\rm reg}(I(G)^s)=2s+\nu(G)-1$$ holds for every integer $s\geq 1$ and every whiskered cycle graph $G$.  Alilooee, Beyarslan, and Selvaraja \cite{abs} extend this result to any unicyclic graph which is not a cycle.

Let $s\geq 1$ be an integer. Jayanthan, Narayanan and Selvaraja \cite[Corollary 5.1]{av} prove that the equality ${\rm reg}(I(G)^s)=2s+\nu(G)-1$ is true for several subclasses of bipartite graphs. The most interesting part of \cite[Corollary 5.1]{av} states that the equality ${\rm reg}(I(G)^s)=2s+\nu(G)-1$ holds for every unmixed bipartite graph. The aim of the current paper is to extend this result as follows. It is obvious that every unmixed bipartite graph (without isolated vertices) is a very well-covered graph. On the other hand, by \cite[Theorem 4.12]{mm}, for every very well-covered graph $G$ we have ${\rm reg}(I(G))= \nu(G)+1$. Therefore, it is natural to ask whether the equality ${\rm reg}(I(G)^s)=2s+\nu(G)-1$ is true for every very well-covered graph. Unfortunately, we are not able to give a general answer to this question. However, we prove in Theorem \ref{th3.11} that if $G$ is a very well-covered graph which has no odd cycle of length at most $2k-1$ ($k\geq 3$), then ${\rm reg}(I(G)^s)=2s+\nu(G)-1$, for every integer $1\leq s\leq k-2$. As bipartite graphs have no odd cycles, our result is an extension of \cite[Corollary 5.1]{av} for unmixed bipartite graphs.


\section{Preliminaries} \label{sec2}

In this section, we recall the definitions and basic facts which are needed in the next section.

Let $G$ be a finite simple graph with vertex set $V(G)$ and edge set $E(G)$. We assume that our graphs have no isolated vertices. By abusing the notation, we sometimes identify the edges of $G$ by quadratic monomials. A subgraph $H$ of $G$ is called {\it induced} if for every $u, v\in V(H)$, we have $\{u,v\}$ is an edge of $H$ if and only if $\{u,v\}$ is an edge of $G$. A {\it matching} in $G$ is a subgraph of $G$ consisting of pairwise disjoint edges. If the subgraph is induced, the matching is called an {\it induced matching}. The largest size of an induced matching in $G$ is called its {\it induced matching number} and is denoted by $\nu(G)$. A {\it perfect matching} of $G$ is a matching which has the same vertex set as $G$.  A subset $X$ of $V(G)$ is called {\it independent} if $\{x,y \}$ is not an edge of $G$ for all $x,y\in X$. A graph $G$ is called {\it unmixed} if all maximal independent sets of $G$ have the same size. The graph $G$ is said to be {\it very well-covered} if it has an even number of vertices and moreover, every maximal independent subset of $G$ has cardinality $|V(G)|/2$. In particular, every very well-covered graph is unmixed.
		
A sequence$$W: v_1, v_2, \ldots, v_m$$ of vertices of $G$ is called a {\it walk} if $\{v_i, v_{i+1}\}$ is an edge of $G$, for every integer$1\leq i\leq m-1$. If moreover $v_1=v_m$, then we say that $W$ is a {\it closed walk}. The {\it length} of a walk, path, or cycle is the number of its edges. A walk (resp. cycle) is called an {\it odd walk} (resp. {\it odd cycle}) if its length is an odd integer. It is easy to see that if $G$ has a closed odd walk, then it also has an odd cycle. We denote length of the shortest odd cycle in $G$ by ${\rm odd-girth}(G)$. We set ${\rm odd-girth}(G)=\infty$ if $G$ has no odd cycle (i.e., $G$ is a bipartite graph).

Our method for proving the main result is based on the recent work of Banerjee \cite{ab}.
We recall the following definition and theorem from \cite{ab}.

\begin{dfn}
Let $G$ be a graph. Two vertices $u$ and $v$ ($u$ may be equal to $v$) are said to be even-connected with respect to an $s$-fold product $e_1 \ldots e_s$ of edges of $G$, if there is a path $p_0, p_1, \ldots, p_{2l+1}$, $l\geq 1$ in $G$ such that the following conditions hold.
\begin{itemize}
\item[(i)] $p_0=u$ and $p_{2l+1}=v$.

\item[(ii)] For all $0\leq k\leq l-1, \{p_{2k+1},p_{2k+2}\}=e_i$ for some \emph{i}.

\item[(iii)] For all \emph{i}, $\mid \{k\mid \{p_{2k+1},p_{2k+2}\}=e_i\}\mid \leq\mid\{j\mid e_i=e_j\}\mid$.
\end{itemize}
\end{dfn}

\begin{thm}\cite[Theorems 6.1 and 6.7]{ab} \label{increase}
Assume that $s\geq 1$ is an integer, $G$ is a graph and $I=I(G)$ is its edge ideal. Let $m$ be a minimal generator of $I^s$. Then the ideal $(I^{s+1}:m)$ is generated by monomials of degree two and for every generator $uv$ ($u$ may be equal to $v$) of this ideal, either $\{u,v\}$ is an edge of $G$ or $u$ and $v$ are even-connected with respect to $m$.
\end{thm}

We also need to remind the following characterization of very well-covered graphs, which is known by \cite[Theorem 1.2]{f}.

\begin{thm} \label{verywell}
A graph $G$ is very well-covered if and only if it has a perfect matching $M$ with the following properties.
\begin{itemize}
\item[(i)] No edge of $M$ belongs to a triangle of $G$.
\item[(ii)] If an edge of $M$ is the central edge of a path of length $3$, then the two vertices at the ends of the path must be adjacent.
\end{itemize}
\end{thm}


\section{Main results} \label{sec3}

In this section, we prove the main result of this paper. Namely, we show in Theorem \ref{th3.11} that if $G$ is a very well-covered graph which has no odd cycle of length at most $2k-1$ ($k\geq 3$), then ${\rm reg}(I(G)^s)=2s+\nu(G)-1$, for every integer $1\leq s\leq k-2$. As we mention in introduction, it is known by \cite[Theorem 4.5]{sb} that$${\rm reg}(I(G)^s)\geq 2s+\nu(G)-1.$$Thus, we only need to prove the converse inequality. The main tool in our proof is the following inequality due to Banerjee \cite[Theorem 5.2]{ab}.
\[
\begin{array}{rl}
{\rm reg}(I(G)^{s+1})\leq \max\{{\rm reg}(I(G)^{s+1}: m_l)+2s, 1\leq l\leq r, {\rm reg} (I(G)^s)\},
\end{array} \tag{$\ast\ast$} \label{astast}
\]
where $\{m_1, \ldots, m_r\}$ is the set of minimal monomial generators of $I(G)^s$. In order to use the above inequality, we need to study the regularity of ideals $(I(G)^{s+1}: m_l)$. The following proposition is the first step in our study. It asserts that if $G$ has no odd cycle of length at most $2k-1$, then for every integer $s$ with $1\leq s\leq k-1$, the ideal $(I(G)^{s+1}: m_l)$ is a quadratic squarefree monomial ideal, for every $1\leq l\leq r$. This shows that $(I(G)^{s+1}: m_l)$ is the edge ideal of a graph. We then conclude in Corollary \ref{odgi} that the graph associated to $(I(G)^{s+1}: m_l)$ has no odd cycle of length at most $2(k-s)-1$. Notice that Proposition \ref{th3.1} and Corollary \ref{odgi} generalize \cite[Proposition 3.5]{aa}.

\begin{prop} \label{th3.1}
Let $G$ be a graph with ${\rm odd-girth}(G)\geq 2k+1$ ($k\geq 2$) and let $1\leq s\leq k-1$ be an integer. Then for any $s$-fold product $e_1\ldots e_s$ of edges of $G$ (with the possibility of $e_i$ being as same as $e_j$ for $i\neq j$), the ideal $(I(G)^{s+1}: e_1e_2\ldots e_s)$ is a quadratic squarefree monomial ideal. Moreover, for an edge $e\in E(G)$, let $G'$ be the graph with $I(G')=(I(G)^2:e)$. Then ${\rm odd-girth}(G')\geq 2k-1$.
\end{prop}

\begin{proof}
We know from Theorem \ref{increase} that $(I(G)^{s+1}: e_1e_2\ldots e_s)$ is generated by monomials of degree two. Hence, we prove that $(I(G)^{s+1}: e_1e_2\ldots e_s)$ is a squarefree monomial ideal.
	
By contradiction, suppose that for a vertex $u\in V(G)$, we have$$u^2\in (I(G)^{s+1}: e_1e_2\ldots e_s).$$It follows from Theorem \ref{increase} that $u$ is even-connected to itself with respect to $e_1\ldots e_s$. Thus, there is a path $u=p_0, p_1, \ldots, p_{2l}, p_{2l+1}=u$ in $G$, for some integer $l\leq s$. But this is an odd closed walk of length$$2l+1\leq 2s+1\leq 2k-1,$$ which implies that $G$ has an odd cycle of length at most $2k-1$. This is a contradiction, as ${\rm odd-girth}(G)\geq 2k+1$. Thus, $(I(G)^{s+1}: e_1\ldots e_s)$ is a squarefree monomial ideal.
	
We now prove the last part of the proposition. Suppose that $e=\{x, y\}$. Let $\ell\leq k-1$ be an integer. We must prove that $G'$ has no odd cycle of length $2\ell-1$. Assume by contradiction that$$C: w_1, w_2, \ldots, w_{2\ell-1}$$is an odd cycle of length $2\ell-1$ in $G'$. Since $G$ has no odd cycle of length $2\ell-1$, at least one of the edges of $C$ belongs to $E(G')\setminus E(G)$. First assume that $C$ has only one edge, say $\{w_1, w_2\}$ in $E(G')\setminus E(G)$. We conclude that $w_1$ and $w_2$ are even-connected with respect to $e$. Therefore, we have the following closed walk of length $2\ell+1$ in $G$.$$w_1, x, y, w_2, w_3, \ldots, w_{2\ell-1}$$This means that $G$ has an odd cycle of length at most $2\ell+1\leq 2k-1$ which is a contradiction.
	
Next, suppose that $C$ has at least two edges, say $\{w_i, w_{i+1}\}$ and $\{w_j, w_{j+1}\}$ in $E(G')\setminus E(G)$ (where by $w_{m+1}$ we mean $w_1$). Hence, $w_i$, $w_{i+1}$ and $w_j$, $w_{j+1}$ are even-connected with respect to $e$. We may assume that $i< j$. As $w_i$ and $w_{i+1}$ are even-connected with respect to $e$, we suppose without loss of generality that $G$ has the path $w_i, x, y, w_{i+1}$. On the other hand, since $w_j$ and $w_{j+1}$ are even-connected with respect to $e$, it follows that $G$ contains either of the paths$$Q_1: w_j, x, y, w_{j+1} \ \ \ \ \ {\rm or} \ \ \ \ \ Q_2: w_j, y, x,  w_{j+1}.$$If $G$ has the path $Q_1$, then we have the following closed odd walk in $G$.$$W: w_1, w_2, \ldots, w_i, x, y, w_{i+1}, w_{i+2}, \ldots, w_j, x, y, w_{j+1}, w_{j+2}, \ldots, w_{2\ell-1}$$Notice that $W$ is the union of the following closed walks.$$W_1: w_1, \ldots, w_i, x, y, w_{j+1}, \ldots, w_{2\ell-1}, \ \ \ \ \ \ W_2: w_{i+1}, w_{i+2}, \ldots, w_j, x, y$$As $W$ has odd length, one of $W_1$ and $W_2$ has odd length too. This means that $G$ has an odd cycle of length at most $2\ell+1\leq 2k-1$ which is a contradiction and completes the proof. Hence, assume that $G$ contains the path $Q_2$. Then we have the following closed odd walk in $G$.$$W': w_1, w_2, \ldots, w_i, x, y, w_{i+1}, w_{i+2}, \ldots, w_j, y, x, w_{j+1}, w_{j+2}, \ldots, w_{2\ell-1}$$Clearly, $W'$ is the union of the following closed walks.$$W'_1: w_1, \ldots, w_i, x, y, x, w_{j+1}, w_{j+2}, \ldots, w_{2\ell-1}, \ \ \ \ \ \ W'_2: w_{i+1}, w_{i+2}, \ldots, w_j, y$$Since $W'$ has odd length, one of $W'_1$ and $W'_2$ has odd length too. Therefore, $G$ has an odd cycle of length at most $2\ell+1\leq 2k-1$ which is again a contradiction and completes the proof.
\end{proof}

The following lemma will be used in the proof of Corollary \ref{odgi}, in order to study the ${\rm odd-girth}$ of the graph associated to $(I(G)^s:e_1, \dots,e_s)$. 

\begin{lem}\label{lem3.2}
Let $G$ be a graph with ${\rm odd-girth}(G)\geq 2k+1$ ($k\geq 2$) and let $I=I(G)$ be the edge ideal of $G$. Assume that $1\leq s\leq {k-1}$ is an integer and $e_1\dots e_s$ is an $s$-fold product of edges of $G$. Then for every integer $1\leq i\leq s$ we have$$(I^{s+1}:e_1\ldots e_s)=((I^2:e_i)^s:\Pi_{i\neq j}e_j).$$
\end{lem}	

\begin{proof}
Without loss of generality, suppose that $i=1$ and $e_1=\{x, y\}$. We first prove that$$(I^{s+1}:e_1\ldots e_s)\subseteq ((I^2:e_1)^s: e_2\ldots e_s).$$

We know from Theorem \ref{increase} that $(I^{s+1}:e_1\ldots e_s)$ is a generated by quadratic monomials $uv$, where $uv\in I$ or $u$ and $v$ are even-connected with respect to $e_1\ldots e_s$. If $uv\in I$ then $uve_1\in I^2$. Thus, $uv\in (I^2:e_1)$. On the other hand, for any $2\leq j\leq s$, we have $e_j\in (I^2:e_1)$. Therefore, $uve_2e_3\ldots e_s\in (I^2:e_1)^s$. This shows that$$uv \in ((I^2:e_1)^s:e_2\ldots e_s).$$Hence, assume that $uv\notin I$ and therefore, $u$ and $v$ are even-connected with respect to $e_1\ldots e_s$. Thus, there is a path $P: u=p_0, p_1, \ldots, p_{2l}, p_{2l+1}=v$ in $G$ such that

\begin{itemize}
\item[(i)] for all $0\leq t\leq l-1, \{p_{2t+1},p_{2t+2}\}=e_j$ for some $j$; and
\item[(ii)] for all $j$, $|\{t\mid \{p_{2t+1},p_{2t+2}\}=e_j\}| \leq |\{i\mid e_i=e_j\}|$.
\end{itemize}
	
By Proposition \ref{th3.1}, there is a graph $G'$ with $I(G')=(I^2:e_1)$. Obviously, $G$ is a subgraph of $G'$. If there is no $t\in\{0,1,\ldots,l-1\}$ such that $\{p_{2t+1},p_{2t+2}\}=e_1$, then the path $P$ is an even-connection with respect to $e_2\ldots e_s$ in $G'$. Thus,$$uv\in(I(G')^s: e_2\ldots e_s)=((I^2:e_1)^s: e_2\ldots e_s).$$Therefore, suppose that there is an integer $t \in\{0,1,\ldots,l-1\}$ such that $\{p_{2t+1}, p_{2t+2}\}=e_1$. In other words, $G$ contains the path $p_{2t}, x, y, p_{2t+3}$. Hence, $\{p_{2t}, p_{2t+3}\}\in E(G')$. Thus, we have the following even-connection with respect to $e_2\ldots e_s$ in $G'$.$$u=p_0, p_1, p_2, \ldots, p_{2t-1}, p_{2t}, p_{2t+3}, p_{2t+4}, \ldots,  p_{2l+1}=v$$Therefore,$$uv\in (I(G')^s:e_2\dots e_s)= ((I^2:e_1)^s:e_2\dots e_s).$$This proves the inclusion$$(I^{s+1}:e_1\ldots e_s)\subseteq ((I^2:e_1)^s: e_2\ldots e_s).$$
	
We next prove the other inclusion. As above, assume that $G'$ is the graph with $I(G')=(I^2:e_1)$. We know from Theorem \ref{increase} that $(I(G')^s:e_2\ldots e_s)$ is a generated by quadratic monomials $uv$, where $uv\in I(G')$ or $u$ and $v$ are even-connected with respect to $e_2\ldots e_s$ in $G'$. There is nothing to prove if $uv\in I(G')=(I^2:e_1)$ (since then clearly $uv\in (I^{s+1}:e_1\ldots e_s)$). Therefore, assume that $uv\notin I(G')$. Hence, $u$ and $v$ are even-connected with respect to $e_2\ldots e_s$ in $G'$. Thus, there is a path $Q: u=q_0, q_1, \ldots, q_{2h}, q_{2h+1}=v$ ($h\leq s-1$) in $G'$ such that

\begin{itemize}
\item[(i)] for all $0\leq t\leq h-1, \{q_{2t+1},q_{2t+2}\}=e_j$ for some $j$ with $2\leq j\leq s$; and
\item[(ii)] for all $j$ with $2\leq j\leq s$, we have $|\{t\mid \{q_{2t+1},q_{2t+2}\}=e_j\}| \leq |\{i\mid e_i=e_j\}|$.
\end{itemize}	
If for each $0\leq t \leq h$, $\{q_{2t}, q_{2t+1}\}\in E(G)$, then $Q$ is an even-connection in $G$ and therefore, $uv\in (I^{s+1}:e_1\ldots e_s)$, by Theorem \ref{increase}. So suppose there exists $t$ such that $\{q_{2t}, q_{2t+1}\}\in E(G')\setminus E(G)$. It follows from Theorem \ref{increase} that $q_{2t}$ and $q_{2t+1}$ are even-connected with respect to $e_1$. Hence, $G$ contains the path $q_{2t}, x, y, q_{2t+1}$. If there is no integer $t'\neq t$ with $\{q_{2t'}, q_{2t'+1}\}\in E(G')\setminus E(G)$, then we have the path $$Q':u=q_0, q_1, \ldots, q_{2t}, x, y, q_{2t+1}, \ldots, q_{2h+1}=v$$in $G$, which is  an even-connection between $u$ and $v$ with respect to $e_1\dots e_s$. Therefore, the assertion follows from Theorem \ref{increase}. Thus, assume that there are integers $t, t'\in \{0, 1, \ldots, h\}$ with$$\{q_{2t}, q_{2t+1}\}\in E(G')\setminus E(G) \ \ \ \ \ {\rm and} \ \ \ \ \ \{q_{2t'}, q_{2t'+1}\}\in E(G')\setminus E(G)$$ such that $t<t'$ and $t'-t$ is maximum. Since $q_{2t}$ and $q_{2t+1}$ are even-connected with respect to $e_1$, without loss of generality, we may suppose that we have the path $q_{2t}, x, y, q_{2t+1}$ in $G$. On the other hand, as $q_{2t'}$ and $q_{2t'+1}$ are even-connected with respect to $e_1$, we have either of the paths$$Q_1: q_{2t'}, x, y, q_{2t'+1} \ \ \ \ \ {\rm or} \ \ \ \ \ Q_2: q_{2t'}, y, x, q_{2t'+1}$$ in $G$. If $G$ has the path $Q_1$, then it contains the following path.$$Q'': u=q_0, q_1, \ldots, q_{2t}, x, y, q_{2t'+1}, \ldots, q_{2h+1}=v$$Observer that $Q''$ is an even-connection with respect to $e_1\ldots e_s$. Therefore, Theorem \ref{increase} implies that$$uv\in (I^{s+1}:e_1\ldots e_s),$$and the assertion follows. If $G$ has the path $Q_2$, then it contains the closed walk$$W: y, q_{2t+1}, q_{2t+2}, \ldots, q_{2t'},$$of length$$2(t'-t)+1\leq 2h+1\leq 2(s-1)+1\leq 2(k-2)+1.$$Thus, $G$ contains an odd cycle of length at most $2k-3$ which is a contradiction by ${\rm odd-girth}(G)\geq 2k+1$.
\end{proof}	

We proved in Proposition \ref{th3.1} that if $G$ is a graph with ${\rm  odd-girth}(G)\geq 2k+1$ ($k\geq 2$), then $(I(G)^{s+1}:e_1\dots e_s)$ is a quadratic squarefree monomial ideal, for every integer $1\leq s\leq {k-1}$. Thus, there exists a graph $G'$ with $I(G')=(I(G)^{s+1}:e_1\dots e_s)$. The following corollary shows that $G'$ can not have arbitrarily small ${\rm odd-girth}$.

\begin{cor} \label{odgi}
Let $G$ be a graph with ${\rm  odd-girth}(G)\geq 2k+1$ ($k\geq 2$) and let $1\leq s\leq {k-1}$ be an integer. Assume that $G'$ is the graph with  $I(G')=(I(G)^{s+1}:e_1\dots e_s)$. Then ${\rm odd-girth}(G')\geq 2(k-s)+1$.
\end{cor}

\begin{proof}
We showed in Lemma \ref{lem3.2} that $(I(G)^{s+1}: e_1\ldots e_s)=((I(G)^2:e_1)^s:e_2\ldots e_s)$. Set $I'=(I(G)^2:e_1)$. Thus,
\begin{align*}
& (I(G)^{s+1}: e_1\ldots e_s)=(I'^s: e_2\ldots e_s)=((I'^2:e_2)^{s-1}:e_3\ldots e_s)\\ & = \cdots =(((((I^2:e_1)^2:e_2)^2:e_3)^2:\ldots )^2:e_s).
\end{align*}
The assertion now follows from Proposition \ref{th3.1} by induction on $s$.
\end{proof}	

The following proposition says that for every very well-covered graph $G$ with ${\rm odd-girth}(G)\geq 7$ and every $e\in E(G)$, the graph associated to $(I(G)^2:e)$ is very well-covered too.

\begin{prop} \label{th3.7}
Let $G$ be a very well-covered graph with ${\rm odd-girth}(G)\geq 7$ and assume that $e$  is an edge of $G$. Suppose that $G'$ is the graph with $I(G')=(I(G)^2:e)$. Then $G'$ is a very well-covered graph.
\end{prop}

\begin{proof}
Assume that $e=\{x, y\}$. As $G$ is a very well-covered covered graph, it has a perfect matching $M$ which satisfies the conditions of Theorem \ref{verywell}. Since $E(G)\subseteq E(G')$, we conclude that $M$ is a perfect matching of $G'$ too. As ${\rm odd-girth}(G)\geq 7$, we conclude from Proposition \ref{th3.1} that ${\rm odd-girth}(G')\geq 5$. Hence, in $G'$, no edge of $M$ is contained in a triangle. Based on Theorem \ref{verywell}, we must prove that if an edge of $M$ is the central edge of a path of length $3$, then the two vertices at the ends of the path are adjacent. Let $e'=\{x', y'\}$ be an edge of $M$ which is the central edge of a path of length $3$, say $z, x', y', w$. If $\{z, x'\}$ and $\{y', w\}$ are edges of $G$, then Theorem \ref{verywell} implies that $z$ and $w$ are adjacent. Thus, we may suppose that at least one of $\{z, x'\}$ and $\{y', w\}$ belong to $E(G')\setminus E(G)$. Assume first that exactly one of $\{z, x'\}$ and $\{y', w\}$, say $\{z, x'\}$ is in $E(G')\setminus E(G)$. Therefore, $G$ has the path $z, x, y, x'$. Since $\{x', y'\}$ and $\{y', w\}$ are edges of $G$, we conclude that $G$ contains that $z, x, y, x', y', w$. Hence, by Theorem \ref{verywell} $w$ and $y$ are adjacent in $G$. Thus, we have that path $z, x, y, w$ in $G$ which means that $z$ and $w$ are adjacent in $G'$.

Next, assume that both $\{z, x'\}$ and $\{y', w\}$ belong to $E(G')\setminus E(G)$. As $\{z, x'\}\in E(G')\setminus E(G)$, we may assume without loss of generality that $G$ has the path $z, x, y, x'$. Thus, $\{z, x\}$ and $\{y, x'\}$ are edges of $G$. On the other hand, since $\{y', w\} \in E(G')\setminus E(G)$, we have either of the paths$$P: y', x, y, w \ \ \ \ \ \ \ {\rm or} \ \ \ \ \ \ \ P': y', y, x, w$$ in $G$. If $G$ contains $P$, then we have the path $z, x, y, w$ in $G$ which means that $\{z, w\}$ is an edge of $G'$. If $G$ contains the path $P'$, then we have a triangle with vertices $x', y, y'$ in $G$. This contradicts the assumption that ${\rm odd-girth}(G)\geq 7$. Therefore, $G$ can not have the path $P'$ and this completes the proof.
\end{proof}

As a consequence of Proposition \ref{th3.7}, we obtain the following result.
		
\begin{cor} \label{cor3.8}	
Let $G$ be a very well-covered graph with ${\rm odd-girth}(G)\geq 2k+1$ ($k\geq 3$) and assume that $e_1, \ldots, e_s$ are (not necessarily distinct) edges of $G$, where $s$ is an integer with $1\leq s\leq k-2$. Suppose that $G'$ is the graph with $I(G')=(I(G)^{s+1}:e_1\dots e_s)$. Then $G'$ is a very well-covered.				 \end{cor}

\begin{proof}
We know from Lemma \ref{lem3.2} that $(I(G)^{s+1}: e_1\ldots e_s)=((I(G)^2:e_1)^s:e_2\ldots e_s)$. Set $I'=(I(G)^2:e_1)$. Thus,
\begin{align*}
& (I(G)^{s+1}: e_1\ldots e_s)=(I'^s: e_2\ldots e_s)=((I'^2:e_2)^{s-1}:e_3\ldots e_s)\\ & = \cdots =(((((I^2:e_1)^2:e_2)^2:e_3)^2:\ldots )^2:e_s).
\end{align*}
The assertion now follows from Corollary \ref{odgi} and Proposition \ref{th3.7} by induction on $s$.
\end{proof}

We are now ready to prove the main result of this paper.

\begin{thm} \label{th3.11}
Let $G$ be a very well-covered graph with ${\rm  odd-girth}(G)\geq 2k+1$ ($k\geq3$). For every integer $s$ with $1\leq s\leq k-2$, we have$${\rm reg}(I(G)^s)=2s+\nu(G)-1.$$
\end{thm}
	
\begin{proof}
We use induction on $s$. For $s=1$, the assertion follows from \cite[Theorem 4.12]{mm}. Now, assume that the desired equality is true for $s$ and we prove it for $s+1$. By induction hypothesis, we know that ${\rm reg}(I(G)^s)=2s+\nu(G)-1$. Let $e_1, \ldots, e_s$ be (not necessarily distinct) edges of $G$ and suppose that $G'$ is the graph with $I(G')=(I(G)^{s+1}:e_1\dots e_s)$. By Corollary \ref{cor3.8}, $G'$ is a very well-covered graph and hence, \cite[Theorem 4.12]{mm} implies that$${\rm reg}((I(G)^{s+1}:e_1\dots e_s))={\rm reg}(I(G')=\nu(G')+1\leq \nu(G)+1,$$where the inequality follows from \cite[Proposition 4.4]{av}. Applying inequality (\ref{astast}) implies that$${\rm reg}(I(G)^{s+1})\leq 2s+\nu(G)+1.$$The converse inequality follows from Theorem \cite[Theorem 4.5]{sb}.
\end{proof}



\end{document}